\documentclass[reqno,a4paper]{amsart}
\usepackage{amsmath}
\usepackage{amssymb,amsfonts}
\usepackage{amsthm}
\usepackage{enumerate}
\usepackage[mathscr]{eucal}
\usepackage{eqlist}

\theoremstyle{plain}
\newtheorem{theorem}{Theorem}[section]

\newtheorem{lemma}{Lemma}[section]

\theoremstyle{definition}

\theoremstyle{remark}

\numberwithin{equation}{section}

\begin{document}

\title[On the Diophantine equation $(x+1)^{k}+(x+2)^{k}+...+(2x)^{k}=y^n$]
{On the Diophantine equation $(x+1)^{k}+(x+2)^{k}+...+(2x)^{k}=y^n$ }

\author{Attila B\'{e}rczes, Istv\'{a}n Pink, Gamze Sava\c{s}, and G\"{o}khan Soydan}

\address{{\bf Attila B\'erczes}\\
	University of Debrecen \\
	Institute of Mathematics\\
	H-4032 Debrecen, Hungary}
\email{berczesa@science.unideb.hu}

\address{{\bf Istv\'{a}n Pink}\\
	University of Debrecen \\
	Institute of Mathematics\\
	H-4032 Debrecen, Hungary}
\email{pinki@science.unideb.hu}

\address{{\bf Gamze Sava\c{s}}\\
Department of Mathematics \\
Uluda\u{g} University\\
 16059 Bursa, Turkey}
\email{gamzesavas91@gmail.com}

\address{{\bf G\"{o}khan Soydan} \\
	Department of Mathematics \\
	Uluda\u{g} University\\
	16059 Bursa, Turkey}
\email{gsoydan@uludag.edu.tr }

\newcommand{\acr}{\newline\indent}

\thanks{}


\begin{abstract}
In this work, we give upper bounds for $n$ on the title equation. Our results depend on assertions describing the precise exponents of $2$ and $3$ appearing in the prime factorization of $T_{k}(x)=(x+1)^{k}+(x+2)^{k}+...+(2x)^{k}$. Further, on combining Baker's method with the explicit solution of polynomial exponential congruences (see e.g. \cite{BHMP}), we show that for $2 \leq x \leq 13, k \geq 1, y \geq 2$ and $n \geq 3$ the title equation has no solutions.
\end{abstract}

%\dedicatory{Dedicated to K\'{a}lm\'{a}n Gy\H{o}ry on the occasion of his 76th birthday}
\maketitle

%################################%
\section{Introduction}\label{sec:1}
%################################%

Let $x$ and $k$ be positive integers.  Write
\begin{equation*}
S_{k}(x)= 1^k + 2^k + ... +x^k
\end{equation*}
for the sum of the $k-th$ powers of the first $x$ positive integers. The Diophantine equation

\begin{equation}  \label{eq:1.1}
 S_{k}(x)= y^{n} \:,
\end{equation}
in unknown positive integers $k,n,x,y$ with $n\geq 2$ has a rich history. In $1875$, the classical question of Lucas \cite{Lu} was whether equation \eqref{eq:1.1} has only the solutions $x=y=1$ and $x=24$, $y=70$ for $(k,n)=(2,2)$. In $1918$, Watson \cite{Wat} solved equation \eqref{eq:1.1} with $(k,n)=(2,2)$. In $1956$,  Sch\"{a}ffer \cite{Sch} considered  equation \eqref{eq:1.1}. He showed, for fixed  $k\geq 1$ and $n\geq 2$, that \eqref{eq:1.1} possesses at most finitely many solutions in positive integers $x$ and $y$, unless

\begin{equation}  \label{eq:1.2}
(k,n) \in \{(1,2),(3,2),(3,4),(5,2)\}
\end{equation}
where, in each case, there are infinitely many such solutions.  There are several effective and ineffective results concerning  equation \eqref{eq:1.2}, see the survey paper \cite{GP}. Sch\"{a}ffer's conjectured that \eqref{eq:1.2} has the unique non-trivial (i.e. $(x,y)\neq (1,1)$) solution, namely $(k,n,x,y)=(2,2,24,70)$. In $2004$, Jacobson, Pint\'{e}r, Walsh \cite{JPW} and Bennett, Gy\H{o}ry, Pint\'{e}r \cite{BGP}, proved that the Sch\"{a}ffer's conjecture is true if $2\leq k\leq 58$, $k$ is even $n=2$ and $ 2\leq k \leq 11 $, $n$ is arbitrary, respectively. In $2007$, Pint\'{e}r \cite{P}, proved that the equation

\begin{equation} \label{eq:1.3}
S_{k}(x)= y^{2n}, \ \mbox{in positive integers} \ \ x,y,n \ \mbox{with} \ n>2
\end{equation}
has only the trivial solution $(x,y)=(1,1)$ for odd values of $k$, with $1\leq k<170.$\\

In $2015$, Hajdu \cite{H}, proved that Sch\"{a}ffer's conjecture holds under certain assumptions on $x$, letting all the other  parameters free. He also proved that the conjecture is true if $x \equiv 0,3 \pmod{4}$ and $x<25$. The main tools in the proof of this result were the $2$-adic valuation of $S_{k}(x)$ and local methods for polynomial-exponential congruences. Recently B\'{e}rczes, Hajdu, Miyazaki and Pink \cite{BHMP}, provided all solutions of equation \eqref{eq:1.1} with $1\leq x<25$  and $n\geq 3$.\\

Now we consider the Diophantine equation

\begin{equation}  \label{eq:1.4}
(x+1)^k+(x+2)^k+ ... + (x+d)^k=y^{n}
\end{equation}\
for fixed positive integers $k$ and $d$.\\

In $2013$, Zhang and Bai \cite{ZB}, considered the Diophantine equation \eqref{eq:1.4} with $k=2$. They first proved that all integer solutions of equation \eqref{eq:1.4} such that $n>1$ and $d=x$ are $(x,y)=(0,0)$, $(x,y,n)=(1, \pm 2, 2), (2, \pm 5, 2), (24,\pm 182, 2)$ or $(x,y)=(-1,-1)$ with $2\nmid n$.  Secondly, they showed that if $p\equiv\pm 5 \pmod{12}$ is prime, $p\mid d$ and $v_{p}(d)  \not\equiv 0 \pmod{n}$, then equation \eqref{eq:1.4} has no integer solution $(x, y)$ with $k=2$. In $2014$, the equation
\begin{equation}\label{eq:1.44}
(x-1)^k+x^k+(x+1)^k=y^n  \quad x, y, n \in\mathbb{Z},  \quad n \geq 2,
\end{equation}
was solved completely by Zhang \cite{Zh}, for $k=2, 3, 4$ and the next year, Bennett, Patel and Siksek \cite{BPS},  extend Zhang's result, completely solving equation \eqref{eq:1.44} in the cases $k=5$ and $k= 6$. In 2016, Bennett, Patel and Siksek \cite{BPS2}, considered the equation \eqref{eq:1.4}. They gave the integral solutions to the equation \eqref{eq:1.4} using linear forms in logarithms, sieving and Frey curves where $k=3$, $2\leq d \leq 50,$ $x\geq1$ and $n$ is prime.\

Let $k\geq2$ be even, and let $r$ be a non-zero integer. Recently, Patel and Siksek \cite{PS}, showed that
for almost all $d\geq2$ (in the sense of natural density), the equation
\begin{equation*}
x^k +(x+r)^k +...+(x+(d-1)r)^k =y^n, \quad x, y, n \in\mathbb{Z},  \quad n \geq 2
\end{equation*}
has no solutions. Let $k, l\geq2$ be fixed integers. More recently, Soydan \cite{S}, considered the equation
\begin{equation}\label{eq:1.444}
(x+1)^k+(x+2)^k+...+(lx)^k =y^n, \quad x, y\geq 1, n \in\mathbb{Z},  \quad n \geq 2
\end{equation}
in integers. He proved that the equation \eqref{eq:1.444} has only finitely many solutions in positive integers, $x,y,k,n$  where $l$ is even, $n\geq 2$ and $k\neq 1,3$. He also showed that the equation \eqref{eq:1.444} has infinitely many solutions where $n\geq 2$, $l$ is even and $k=1,3$.\\

In this paper, we are interested in the integer solutions of the equation
\begin{equation}  \label{eq:1.5}
 T_{k}(x)=y^{n}
\end{equation}
where
\begin{equation}  \label{eq:1.6}
T_{k}(x)=(x+1)^k+(x+2)^k+ ... + (2x)^k
\end{equation}
for positive integer $k$. We provide upper bounds for $n$ and give some results about equation \eqref{eq:1.5}.

%################################%
\section{The main results}\label{sec:2}
%################################%

Our main results provide upper bounds for the exponent $n$ in equation \eqref{eq:1.5} in terms of $2$ and $3$-valuations  $v_{2}$ and  $v_{3}$ of some functions of $x$ and $x,k$. Further, on combining Theorem \ref{theo:2.1} with Baker's method and with a version of the local method (see e.g. \cite{BHMP}), we show that for $2 \leq x \leq 13, k \geq 1, y \geq 2$ and $n \geq 3$ equation \eqref{eq:1.5} has no solutions.

For a prime $p$ and an integer $m$, let $v_{p}(m)$ denote the highest exponent $v$ such that $p^{v}|m$.

\begin{theorem}  \label{theo:2.1}
$(i)$ Assume first that  $x \equiv 0 \pmod{4}$. Then for any solution $(k,n,x,y)$ of equation \eqref{eq:1.5}, we get

\begin{equation*}
n\leq \left\{
\begin{array}{ll}
v_{2}(x)-1, & \textrm{ if $k=1$ or $k$ is even,} \\
2v_{2}(x)-2, & \textrm{ if $k\geq 3$ is odd.} \\
\end{array}
\right.
\end{equation*}

$(ii)$ Assume that $x \equiv 1 \pmod{4}$ and $k=1$, then for any solution $(k,n,x,y)$ of equation \eqref{eq:1.5}, we get $n\leq v_{2}(3x+1)-1$.\\

Suppose next that $x \equiv 1,5 \pmod{8}$ and  $x \not\equiv 1 \pmod{32}$ with $k\neq 1$. Then for any solution $(k,n,x,y)$ of equation \eqref{eq:1.5}, we get
\begin{equation*}
n\leq \left\{\def\arraystretch{1.2}
\begin{array}{@{}c@{\quad}l@{}}
v_{2}(7x+1)-1, & \text{if $x \equiv 1\pmod{8}$ and k=2,}\\
v_{2}((5x+3)(3x+1))-2, & \text{if $x \equiv 1 \pmod{8}$ and $k=3$,}\\
v_{2}(3x+1), & \text{if $x \equiv 5 \pmod{8}$ and $k\geq 3$ is odd,}\\
$1$, & \text{if $x \equiv 5 \pmod{8}$ and $k\geq 2$ is even,} \\
$2$, & \text{if $x \equiv 9 \pmod{16}$ and $k\geq 4$ is even,}\\
$3$, & \text{if $x \equiv 9 \pmod{16}$ and $k\geq 5$ is odd}\\
& \text{or}\\
& \text{if $x \equiv 17 \pmod{32}$ and $k\geq 4$ is even,} \\
$4$, & \text{if $x \equiv 17 \pmod{32}$ and $k\geq 5$ is odd.} \\
\end{array}\right.
\end{equation*}

$(iii)$ Suppose now that  $x \equiv 0 \pmod{3}$ and $k$ is odd or $x \equiv 0,4 \pmod{9}$ and $k\geq 2$ is even. Then for any solution $(k,n,x,y)$ of equation \eqref{eq:1.5},
\begin{equation*}
n\leq \left\{\def\arraystretch{1.2}
\begin{array}{@{}c@{\quad}l@{}}
v_{3}(x), & \text{if $x \equiv 0 \pmod{3}$ and $k=1$,}\\
v_{3}(x)-1, & \text{if $x \equiv 0 \pmod{9}$ and $k\geq 2$ is even,}\\
v_{3}(kx^2), & \text{if $x \equiv 0 \pmod{3}$ and $k>3$ is odd,}\\
v_{3}(x^2(5x+3)), & \text{if $x \equiv 0 \pmod{3}$ and $k=3$,} \\
v_{3}(2x+1)-1, & \text{if $x \equiv 4 \pmod{9}$ and $k\geq 2$ is even.} \\
\end{array}\right.
\end{equation*}
\end{theorem}
%\begin{remark} \label{rem.2.2}
%Since the cases $x \equiv 1 \pmod{32}$ and $k\neq 1$ in Theorem \ref{theo:2.1}. $(ii)$ and  $x \equiv 5 \pmod{9}$ and  $k\geq 2$ is even in %Theorem \ref{theo:2.1} $(iii)$ can not be formulated with $v_{2}$ and $v_{3}$ functions, one can not provide upper bounds for the exponent $n$ in %\eqref{eq:1.5} in terms of $v_{2}$ and $v_{3}$ functions of $x$.
%\end{remark}
\begin{theorem}  \label{theo:2.3}
Assume that $x \equiv 1,4 \pmod{8}$ or $x \equiv 4,5 \pmod{8}$. Then Eq. \eqref{eq:1.5} has no solution with $k=1$ or $k\geq 2$ is even, respectively.
\end{theorem}

\begin{theorem} \label{theo:numerical}
Consider equation \eqref{eq:1.5} in positive integer unknowns $(x,k,y,n)$ with
$2 \leq x \leq 13, k \geq 1, y \geq 2$ and $n \geq 3$.
Then equation \eqref{eq:1.5} has no solutions.
\end{theorem}

%################################%
\section{Auxiliary results}\label{sec:3}
%################################%
\subsection{Bernoulli polynomials}
The Bernoulli polynomials $B_{q}(x)$ are defined by
\begin{equation*}
\frac{ze^{zx}}{e^{z}-1}=   \sum_{q=0}^{\infty}\frac{B_{q}(x)z^{q}}{q!}, \quad |z|<2\pi.
\end{equation*} Their expansion around the origin is given by
\begin{equation}\label{eq:3.6}
B_{q}(x)=\sum_{i=0}^{q}{\binom{q}{i}}B_{i}x^{q-i},
\end{equation}
where $B_{n}=B_{n}(0)$ for $(n=0,1,2,...)$ are the Bernoulli numbers. For the following properties of Bernoulli Polynomials, we refer to Haynsworth and Goldberg, \cite{AS}, pp. 804-805 (see also Rademacher, \cite{Rd}):
\begin{equation} \label{eq:3.7}
B_{k}=B_{k}(0), \quad  k=0,1,2,...
\end{equation}

\begin{equation} \label{eq:3.8}
B_{2k+1}(0)=B_{2k+1}(1)=B_{2k+1}=0, \quad k=1,2,...
\end{equation}

\begin{equation} \label{eq:3.9}
B_{k}(1-x)=(-1)^{k}B_{k}(x)
\end{equation}

\begin{equation} \label{eq:3.10}
B_{k}(x)+ B_{k}(x+\frac{1}{2})= 2^{1-k}B_{k}(2x)
\end{equation}

\begin{equation} \label{eq:3.11}
(-1)^{k+1}B_{2k+1}(x)> 0, \quad k=1,2,... \quad 0<x<\frac{1}{2}
\end{equation}

The polynomials $S_{k}(x)$ are strongly connected to the Bernoulli polynomials since $S_{k}(x)$ may be expressed as
\begin{equation}\label{eq:3.12}
S_{k}(x)=\frac{1}{k+1}(B_{k+1}(x+1)-B_{k+1}(0)).
\end{equation}

\subsection{Decomposition of the polynomials $S_k(x)$ and $T_k(x)$}

We start by stating some well-known properties of the polynomial $S_k(x)$ which we will need later; see e.g. \cite{Rd} for details.\\

If $k=1$, then $S_{1}(x)=\frac{x(x+1)}{2}$ while, if $k>1$, we can write
\begin{equation*}
S_{k}(x) = \left\{ \def\arraystretch{1.2}
\begin{array}{@{}c@{\quad}l@{}}
\frac{1}{C_{k}}x^{2}(x+1)^{2}R_{k}(x), & \text{if $k>1$ is odd,}\\
\frac{1}{C_{k}}x(x+1)(2x+1)R_{k}(x), & \text{if $k>1$ is even.}\\
\end{array}\right.
\end{equation*}
where $C_{k}$ is a positive integer and $R_{k}(x)$ is a polynomial with integer coefficients.\\
\noindent By $T_k(x)=S_k(2x)-S_k(x)$ the polynomial $T_k(x)$ is also in a strong connection with the Bernoulli polynomials.
This connection is shown in the below Lemma.
\begin{lemma}  \label{lem:3.1}
\begin{equation}  \label{eq:3.13}
T_k(x)= \frac{B_{k+1}(2x+1)- B_{k+1}(x+1)}{k+1}
\end{equation}
\end{lemma} where $B_q(x)$ is the $q$-th Bernoulli polynomial defined by \eqref{eq:3.6}.

\begin{proof}
It is an application of the equality
\begin{equation*}
\sum_{n=M}^{N-1}n^{k}=\frac{1}{k+1}\{B_{k+1}(N)-B_{k+1}(M)\}
\end{equation*} which is given by Rademacher in \cite{Rd}, pp. 3-4.
\end{proof}
Secondly, applying Lemma \ref{lem:3.1} to equation \eqref{eq:1.5}, we have the following:

\begin{lemma}  \label{lem:3.2}
If $k=1$, then $T_{1}(x)= \frac{x(3x+1)}{2}$, while for $k>1$ we can write\\

$(i)$ $T_{k}(x)=\frac{1}{D_{k}}x(2x+1)M_{k}$, \quad if $k\geq 2$ is even,\\

$(ii)$ $T_{k}(x)=\frac{1}{D_{k}}x^2(3x+1)M_{k}$, \quad if $k>1$ is odd \\ \\
where $D_{k}$ is a positive integer and $M_{k}(x)$ is a polynomial with integer coefficients.
\end{lemma}

\begin{proof}
$(i)$ Firstly we prove that $x=0$ and $x=- \frac{1}{2}$ are roots of the polynomial $T_{k}(x)$ where $k\geq 2$ is even. By \eqref{eq:3.13}, we have
\begin{equation} \label{eq:3.14}
T_{k}(x)= \frac{B_{k+1}(2x+1)-B_{k+1}(x+1)}{k+1}.
\end{equation} It is clear that $x=0$ and $x=- \frac{1}{2}$ satisfy \eqref{eq:3.14} by using \eqref{eq:3.8} and \eqref{eq:3.9}.\\

Secondly we show that $x=0$ and $x=- \frac{1}{2}$ are simple roots of $T_k(x)$ for $k \geq 2$ even. Since for the Bernoulli polynomials $B_n(x)$ we have
\begin{equation*}
 \frac{dB_{n}(x)}{dx}=nB_{n-1}(x)
\end{equation*}

we may write

\begin{equation} \label{eq:3.15}
T'_{k}(x)=(k+1)({2B_{k}(2x+1)-B_{k}(x+1)})
\end{equation}
and
\begin{equation} \label{eq:3.16}
T''_{k}(x)=k(k+1)({4B_{k-1}(2x+1)-B_{k-1}(x+1)}).
\end{equation}

\vskip.2cm\noindent If $k \geq 2$ is even, then $T'_k(0)=(k+1)(2B_k(1)-B_k(1))=(k+1)B_k(1) \ne 0$. So $x=0$ is a simple root of $T_{k}(x)$. Similarly, since
$T'_k(- \frac{1}{2})=(k+1)(2B_k(0)-B_k(\frac{1}{2})) \ne 0$
% $T_{k}(- \frac{1}{2})=0\neq T'_{k}(- \frac{1}{2})$ with \eqref{eq:3.7} and \eqref{eq:3.10}, $x=-\frac{1}{2}$
it follows that $x=-\frac{1}{2}$ is the simple root of $T_{k}(x)$ where $k$ is even.\\

$(ii)$ Now by \eqref{eq:3.9} and \eqref{eq:3.14} we see that $x=-\frac{1}{3}$ is a root of $T_{k}(x)$ whenever $k>1$ is odd.
Using \eqref{eq:3.9}, \eqref{eq:3.11}, \eqref{eq:3.13} and \eqref{eq:3.14}, we have $T_{k}(-\frac{1}{3})=0\neq T'_{k}(- \frac{1}{3})$. So $x=-\frac{1}{3}$ is a simple root of $T_{k}(x)$ where $k$ is odd. Similarly we can show that $x=0$ is a double root of $T_{k}(x)$ if $k$ is odd. So, the proof is completed.
\end{proof}

\subsection{Congruence properties of $S_k(x)$}

In this subsection we give some useful lemmas which will be used to prove some of our main results.

\begin{lemma}(\cite{ST}, Lemma 1)  \label{lem:3.3}
If $p$ is a prime, $d, \ q \in \mathbb{N}$, $k \in \mathbb{Z^{+}}$, $m_{1} \in p^d\mathbb{N} \cup \{0\}$ and $m_{2} \in p^d\mathbb{N} \cup \{0\}$, then
\begin{equation}\label{eq:3.17}
S_{k}(qm_{1}+m_{2}) \equiv qS_{k}(m_{1})+ S_{k}(m_{2}) \pmod{p^{d}}.
\end{equation}
\end{lemma}
\begin{proof}
The proof is similar to Lemma 1 in \cite{ST}.
\end{proof}
\begin{lemma}(\cite{H}, Lemma 3.2)\label{lem:3.4}
Let $x$ be a  positive integer. Then we have
\begin{equation*}
v_{3}(S_{k}(x))= \left\{ \def\arraystretch{1.2}%
\begin{array}{@{}c@{\quad}l@{}}
v_{3}(x(x+1)), & \text{if $k=1$,}\\
v_{3}(x(x+1)(2x+1))-1, & \text{if $k$ is even,}\\
0, & \text{if $x \equiv 1 \pmod{3}$ and $k\geq 3$ is odd,}\\
v_{3}(kx^2(x+1)^2)-1, & \text{if $x \equiv 0,2 \pmod{3}$ and $k\geq 3$ is odd.} \\
\end{array} \right.
\end{equation*}
\end{lemma}

\begin{lemma}(\cite{ST}, Theorem 3)\label{lem:3.5}
Let $p$ be an odd prime and let $m$ and $k$ be  positive integers.\\

$(i)$ For some integer $d\geq 1$, we can write
\begin{equation*}
m=qp^d+r \frac{p^d-1}{p-1}=qp^d+ rp^{d-1}+ rp^{d-2}+ \cdots +rp^0 ,
\end{equation*}
where $r \in \{0,1,...,p-1\}$ and $0 \leq q \not\equiv r \equiv m \pmod{p}.$\\

$(ii)$ In the case of $m \equiv 0 \pmod{p}$, we have
\begin{equation*}
S_{k}(m) \equiv
\left\{\def\arraystretch{1.2}
\begin{array}{@{}c@{\quad}l@{}}
-p^{d-1} \pmod{p^d}, & \text{if $p-1\mid k$,}\\
0 \pmod{p^d}, & \text{if $p-1 \nmid k$.}\\
\end{array}\right.
\end{equation*}

$(iii)$ In the case of $m \equiv -1 \pmod{p}$, we have
\begin{equation*}
S_{k}(m) \equiv
\left\{ \def\arraystretch{1.2}
\begin{array}{@{}c@{\quad}l@{}}
-p^{d-1}(q+1) \pmod{p^d}, & \text{if $p-1\mid k$,}\\
0 \pmod{p^d}, & \text{if $p-1 \nmid k$.}\\
\end{array}\right.
\end{equation*}

$(iv)$ In the case of $m \equiv \frac{p-1}{2} \pmod{p}$, we have
\begin{equation*}
S_{k}(m) \equiv
\left\{ \def\arraystretch{1.2}
\begin{array}{@{}c@{\quad}l@{}}
-p^{d-1}(q+\frac{1}{2} ) \pmod{p^d}, & \text{if $p-1\mid k$,}\\
0 \pmod{p^d}, & \text{if $p-1 \nmid k$.}\\
\end{array}\right.
\end{equation*}
\end{lemma}

\subsection{Linear forms in logarithms}
For an algebraic number $\alpha$ of degree $d$ over $\mathbb{Q}$, we define the {\it absolute logarithmic height} of $\alpha$ by the following formula:
$$
{\rm h}(\alpha)= \dfrac{1}{d} \left( \log \vert a_{0} \vert + \sum\limits_{i=1}^{d} \log \max \bigr\{ 1, \vert \alpha^{(i)}\vert \bigr\} \right),
$$
where $a_{0}$ is the leading coefficient of the minimal polynomial of $\alpha$ over $\mathbb{Z}$, and $\alpha^{(1)}, \alpha^{(2)}, \,...\,, \alpha^{(d)}$ are the conjugates of $\alpha$ in the field of complex numbers.

Let $\alpha_{1}$ and $\alpha_{2}$ be multiplicatively independent algebraic numbers with $\vert \alpha_{1} \vert \ge 1$ and $\vert \alpha_{2} \vert \ge 1$.
Consider the linear form in two logarithms:
$$
\varLambda=b_{2}\log\alpha_{2}-b_{1}\log \alpha_{1},
$$
where $\log \alpha_{1}, \log \alpha_{2}$ are any determinations of the logarithms of $\alpha_{1}, \alpha_{2}$ respectively,
and $b_{1},b_{2}$ are positive integers.

We shall use the following result due to Laurent \cite{LAUR}.

\begin{lemma}[\cite{LAUR}, Theorem 2] \label{laurentlemma}
Let $ \rho $ and $ \mu $ be real numbers with $ \rho > 1 $ and $ 1/3 \leq \mu \leq 1$.
Set
$$
\sigma=\dfrac{1+2\mu-\mu^{2}}{2}, \quad \lambda= \sigma \log \rho.
$$
Let $a_1,a_2$ be real numbers such that
$$
a_{i} \geq \max \left\lbrace 1, \, \rho \vert \log \alpha_{i} \vert - \log \vert \alpha_{i} \vert + 2D{\rm h}(\alpha_{i}) \right\rbrace \ \ \ \ \ (i=1,2),
$$
where
$$
D= \left[  \mathbb{Q}(\alpha_{1}, \alpha_{2} ): \mathbb{Q} \right]/\left[  \mathbb{R}(\alpha_{1}, \alpha_{2} ): \mathbb{R} \right].
$$
Let $h$ be a real number such that
$$
h \geq \max \left\lbrace
D \left( \log \left( \frac{b_{1}}{a_{2}}+ \frac{b_{2}}{a_{1}} \right) + \log \lambda +1.75 \right)+0.06,\,
\lambda, \,
\frac{D \log 2}{2} \right\rbrace.
$$
We assume that
$$
a_{1}a_{2} \geq \lambda^{2}.
$$
Put
$$
H= \dfrac{h}{\lambda}+ \dfrac{1}{\sigma}, \ \
\omega=2+ 2\sqrt{1+ \dfrac{1}{4H^{2}} }, \ \
\theta=\sqrt{1+ \dfrac{1}{4H^{2}} }+ \dfrac{1}{2H}.
$$
Then we have
$$
\log \vert \varLambda \vert \geq -C h'^{\,2} a_{1}a_{2}- \sqrt{\omega \theta} h'- \log \left( C' h'^{\,2} a_{1}a_{2} \right)
$$
with
$$
h'=h+ \dfrac{\lambda}{\sigma}, \quad
C=C_0\dfrac{\mu}{\lambda^{3}\sigma}, \quad
C'=\sqrt{ \dfrac{C \sigma \omega \theta}{\lambda^{3} \mu} },
$$
where
$$
C_0=\left( \dfrac{\omega}{6}+ \dfrac{1}{2} \sqrt{\dfrac{\omega^{2}}{9}+ \dfrac{8\lambda \omega^{5/4} \theta^{1/4}}{3 \sqrt{a_{1}a_{2}}H^{1/2} } + \dfrac{4}{3} \left( \dfrac{1}{a_{1}}+ \dfrac{1}{a_{2}} \right) \dfrac{\lambda \omega }{H} } \,\right)^{2}.
$$
\end{lemma}

\subsection{A Baker type estimate}

Let $A = \{2,3,6,7,10,11\}$ and consider equation \eqref{eq:1.5} with $x \in A$.
The following lemma provides sharp upper bounds for the solutions $n,k$ of the equation \eqref{eq:1.5}
and will be used in the proof of Theorem \ref{theo:numerical}.

\begin{lemma} \label{L_Baker_bounds}
Let $A = \{2,3,6,7,10,11\}$ and consider equation
\eqref{eq:1.5} with $x \in A$ in integer unknowns $(k,y,n)$ with $k \ge 83, y \ge 2$ and $n \ge 3$ a prime.
Then for $y > 4x^2$ we have $n \leq n_0$, for $y > 10^6$ even $n \leq n_1$ holds, and for $y \leq 4x^2$ we have $k \leq k_1$,
where $n_0=n_0(x), n_1=n_1(x)$ and $k_1=k_1(x)$ are given in Table \ref{table1}.

\begin{table}[h]
\centering
\begin{tabular}{|c|c|c|c|}
\hline
$x$ & $n_0 \ (y>4x^2)$ & $n_1 \ (y>10^6)$ & $k_1 \ (y \le 4x^2)$     \\\hline
$2$ & $7,500$    &  $3,200$   &  $45,000$\\\hline
$3$ & $21,000$   & $10,000$   &  $120,000$\\\hline
$6$ & $94,000$   & $53,000$   &  $540,000$\\\hline
$7$ & $128,000$  & $74,200$   &  $740,000$\\\hline
$10$& $253,000$  & $157,000$  &  $1,450,000$\\\hline
$11$& $301,000$  & $190,000$  &  $1,750,000$\\\hline
\end{tabular}
\vspace{0.2cm} \caption{Bounding $n$ and $k$ under the indicated conditions} \label{table1}
\end{table}

\end{lemma}

\begin{proof} In the course of the proof we will always assume that $x \in A$ and we distinguish three cases according to $y>4x^2$, $y>10^6$ or $y \le 4x^2$.

\vskip.2cm\noindent {\it Case I. $y >4x^2$ }

\vskip.2cm\noindent
We may suppose, without loss of generality, that $n$ is large enough, that is
\begin{equation} \label{n_lower1}
n>n_0.
\end{equation}

\noindent Further, by $k \ge 83$ we easily deduce that for every $x \in A$ we have
\begin{equation} \label{skx_upper_1}
(x+1)^k+(x+2)^k+ \cdots +(2x)^k<2 \cdot (2x)^k,
\end{equation}
and
\begin{equation} \label{skx_upper_2}
(x+1)^k+(x+2)^k+ \cdots +(2x-1)^k<2 \cdot (2x-1)^k.
\end{equation}

\noindent Since $y>4x^2$ by \eqref{eq:1.5}, \eqref{skx_upper_1} and $x \ge 2$ we get that
\begin{equation} \label{B_ge_1}
k \ge 2n.
\end{equation}
\noindent Using \eqref{B_ge_1} and the fact that $n$ is odd we may write $k$ in the form
\begin{equation} \label{k_with_B_and_n}
k=Bn+r \ \textrm{with} \ B \ge 1, \ 0 \le \vert r \vert \le \frac{n-1}{2}.
\end{equation}
\noindent We show that in \eqref{k_with_B_and_n} we have $r \neq 0$. On the contrary, suppose $r=0$. Then, using \eqref{eq:1.5} and \eqref{skx_upper_2}
we infer by \eqref{k_with_B_and_n} that
\begin{align*}
2(2x-1)^{k}&>(x+1)^k+(x+2)^k+\cdots +(2x-1)^k=y^n-(2x)^k=y^n-(2x)^{Bn}\\
&=(y-(2x)^B)(y^{n-1}+\cdots + (2x)^{B(n-1)})  \ge (2x)^{B(n-1)}.
\end{align*}
Hence
$$
n< \frac{\log (2x)}{\log \left( \frac{2x}{2x-1} \right)}+ \frac{\log 2}{B\log \left( \frac{2x}{2x-1} \right)}.
$$
This together with $x \le 11$  and $B \ge 1$ implies $n<82$, which contradicts \eqref{n_lower1}. Thus, $r \neq 0$.

\noindent On dividing equation \eqref{eq:1.5}, by $y^n$ we obviously get
\begin{equation}\label{divide1}
1-\frac{(2x)^k}{y^n}=\frac{s}{y^n},
\end{equation}
\noindent where $s=(x+1)^k+\ldots+(2x-1)^k$. Using \eqref{k_with_B_and_n} and \eqref{divide1} we infer that
\begin{equation}\label{divide2}
\left\vert (2x)^r \cdot \left(\frac{(2x)^B}{y}\right)^n-1 \right\vert=\frac{s}{y^n}.
\end{equation}
Put
\begin{equation}\label{lambda}
\varLambda_r =
\begin{cases}
\,\ \, r \log (2x)-n \log{ \frac{y}{(2x)^B}} & \text{if $r>0$},\\
\,|r| \log (2x)-n \log{ \frac{(2x)^B}{y}} & \text{if $r<0$}.
\end{cases}
\end{equation}
\noindent In what follows we find upper and lower bounds for $\log\vert\varLambda_r\vert$. We distinguish two subcases according to
$$1-\frac{(2x)^k}{y^n} \ge 0.795 \ \textrm{or} \ 1-\frac{(2x)^k}{y^n}<0.795, $$
respectively. If $1-\frac{(2x)^k}{y^n} \ge 0.795$ then by \eqref{eq:1.5} and \eqref{skx_upper_1} we immediately obtain a contradiction,
so we may assume that the latter case holds.

\noindent It is well known (see Lemma B.2 of \cite{smart}) that for every $z \in \mathbb{R}$ with $\vert z-1 \vert < 0.795$ one has
\begin{equation} \label{calculus}
\vert \log{z} \vert <2\vert z-1 \vert.
\end{equation}
\noindent On applying inequality \eqref{calculus} with $z=(2x)^k/y^n$ we get by \eqref{divide1}, \eqref{divide2}, \eqref{lambda} and $(2x)^k \ne y^n$ that
\begin{equation}\label{lambda_upper_1}
\vert \varLambda_r \vert<\frac{2s}{y^n}.
\end{equation}
\noindent Observe that \eqref{eq:1.5} implies
\begin{equation}\label{k_with_nyx}
k<\frac{n\log{y}}{\log{2x}}.
\end{equation}
\noindent Thus by \eqref{lambda_upper_1}, \eqref{skx_upper_2} and \eqref{k_with_nyx} we infer that

\begin{align}\label{lambda_upper}
\log |\varLambda_r|
<-\frac{\log \bigr(\frac{2x}{2x-1}\bigr)}{\log(2x)} (\log y)\,n+\log 4.
\end{align}

Next, for a lower bound for $\log |\varLambda_r|$, we shall use Lemma \ref{laurentlemma} with
$$
(\alpha_1, \alpha_2, b_1,b_2)=
\begin{cases}
\ \left(\frac{y}{(2x)^B},2x,n,r\right) & \text{if $r>0$},\\
\ \left(\frac{(2x)^B}{y},2x,n,|r|\right) & \text{if $r<0$}.
\end{cases}
$$
Using \eqref{eq:1.5} and \eqref{skx_upper_1} one can easily check that $\alpha_1>1$ and $\alpha_2>1$. We show that $\alpha_1, \alpha_2$ are multiplicatively independent. Assume the contrary. Then the set of prime factors of $y$ coincides with that of $2x$.
This implies that $y$ must be even. But for $x \in A$ we easily see that $y$ is odd, which is a contradiction, proving that
$\alpha_1$ and $\alpha_2$ are multiplicatively independent.

\noindent Now, we apply Lemma \ref{laurentlemma} for every $x \in A$ with
\begin{equation}\label{rho}
(\rho,\mu)=(7.7,0.57).
\end{equation}
\noindent In what follows we shall derive upper bounds for the quantities
\begin{equation*}
\rho \vert \log \alpha_{i} \vert - \log \vert \alpha_{i} \vert + 2D{\rm h}(\alpha_{i}), \quad (i=1,2)
\end{equation*}
occurring in Lemma \ref{laurentlemma}. Since $D=1$ and $\alpha_2>1$, for $i=2$ we get
\begin{equation} \label{a2_upper}
\rho \vert \log \alpha_{2} \vert - \log \vert \alpha_{2} \vert + 2D{\rm h}(\alpha_{2})=(\rho+1)\log{2x}.
\end{equation}
\noindent For $i=1$ we obtain
\begin{equation} \label{a1_upper}
\rho \vert \log \alpha_{1} \vert - \log \vert \alpha_{1} \vert + 2D{\rm h}(\alpha_{1})<\frac{\rho+1}{2}\log{2x}+2\log{y},
\end{equation}
To verify that \eqref{a1_upper} is valid we shall estimate $\log \alpha_1$ and ${\rm h}(\alpha_1)$ from above, by using equation \eqref{eq:1.5}, i.e. $s+(2x)^{Bn+r}=y^n$.
Observe
$$
{\rm h}(\alpha_1) ={\rm h}\left(\frac{(2x)^{B}}{y}\right) \le  \log \max\{(2x)^B,y \}
=\begin{cases}
\log y & \text{if $r>0$},\\
\log (2x)^B & \text{if $r<0$}.
\end{cases}
$$
If $r>0$, then
$$
\alpha_1^n=\left(\frac{y}{(2x)^{B}}\right)^{n}=(2x)^{r}+\frac{s}{(2x)^{Bn}}=(2x)^{r}\left(1+\frac{s}{(2x)^{k}}\right)<2(2x)^{r} \ (\text{as $s<(2x)^k$}),
$$
so
$$
\log \alpha_1<\frac{\log 2}{n}+\frac{r}{n}\,\log (2x) \le \frac{\log 2}{n}+\frac{n-1}{2n}\,\log (2x),
$$
whence
$$
\rho \vert \log \alpha_{1} \vert - \log \vert \alpha_{1} \vert + 2D{\rm h}(\alpha_{1})<
$$
$$
<\left(\frac{\log 2}{n\log (2x)}+\frac{n-1}{2n}\right)(\rho-1)\,\log (2x)+2\log y \quad
$$
which by \eqref{rho} and $x \ge 2$ clearly implies \eqref{a1_upper}.

\noindent If $r<0$, then
$$
\alpha_1^n=\left(\frac{(2x)^{B}}{y}\right)^{n}=(2x)^{-r}\left(1-\frac{s}{y^{n}}\right)<(2x)^{-r}=(2x)^{|r|},
$$
so
$$
\log \alpha_1<\frac{|r|}{n}\,\log (2x) \le \frac{n-1}{2n}\,\log (2x), $$
and
$$
\log (2x)^B=\log \alpha_1+\log y<\frac{n-1}{2n}\,\log (2x)+\log y,
$$
and we get
$$
\rho \vert \log \alpha_{1} \vert - \log \vert \alpha_{1} \vert + 2D{\rm h}(\alpha_{1}) <
$$
$$
<\left(\frac{n-1}{2n}\,(\rho-1)+\frac{n-1}{n} \right)\,\log (2x)+2\log y,
$$
which by \eqref{n_lower1} again implies \eqref{a1_upper}.

\noindent In view of \eqref{a2_upper} we can obviously take for every $x \in A$
\begin{equation}\label{a2}
a_2=(\rho+1)\log(2x),
\end{equation}
while for the values $a_1$ we use the upper bound occurring in \eqref{a1_upper}.
Namely, we can take $a_1$ as
\begin{equation} \label{a1}
a_1=\frac{\rho+1}{2}\log(2x)+2\log{y} \quad \textrm{if} \quad x \in A
\end{equation}

\noindent Since $\mu=0.57$ we get
\begin{equation} \label{sigma_lam}
\sigma=0.90755 \ \textrm{and} \ \lambda=0.90755\log{\rho},
\end{equation}
whence by \eqref{rho}, \eqref{a2}, \eqref{a1}, \eqref{sigma_lam} and $y>4x^2$ we easily check that for every $x \in A$
$$
a_1a_2>\lambda^2
$$
holds. Now, we are going to derive an upper bound $h$ for the quantity
$$
\max \left\lbrace
D \left( \log \left( \frac{b_{1}}{a_{2}}+ \frac{b_{2}}{a_{1}} \right) + \log \lambda +1.75 \right)+0.06,\,
\lambda, \,
\frac{D \log 2}{2} \right\rbrace.
$$
Using $D=1$, \eqref{rho}, \eqref{a2}, \eqref{a1}, \eqref{sigma_lam} and $y>4x^2$,
for the values of $h$ occurring in Lemma \ref{laurentlemma} we obtain $h=\log{n}+\varepsilon$, with $\varepsilon=\varepsilon(x)$ given in Table \ref{table2}.

%\newpage

% we obtain the following values for $h$ occurring in Lemma \ref{laurentlemma}:

\begin{table}[h]
\centering
\begin{tabular}{|c|c|c|c|c|c|c|}
\hline
$x$ & $2$ & $3$ & $6$ & $7$ & $10$ & $11$ \\\hline
$\varepsilon$ & $0.3560$ & $0.0995$ & $-0.2275$ & $-0.2877$ & $-0.4144$ & $-0.4458$ \\\hline\hline
\end{tabular}
\vspace{0.2cm} \caption{Choosing the parameter $h=\log{n}+\varepsilon$ occurring in Lemma \ref{laurentlemma} if the case $y > 4x^2$} \label{table2}
\end{table}

\noindent Further, by \eqref{n_lower1} we easily check that for the above values of $h$ assumptions of
Lemma \ref{laurentlemma} concerning the parameter $h$ are satisfied. Using \eqref{n_lower1} again we obtain a lower bound for
$H$ and hence upper bounds for $\omega$ and $\theta$. Moreover, using these values of $\omega$ and $\theta$ by \eqref{rho}, \eqref{a2}, \eqref{a1}, \eqref{sigma_lam} and $y>4x^2$ for $x \in A$ we obtain Table \ref{table3}.

\begin{table}[h]
\centering
\begin{tabular}{|c|c|c|c|c|c|c|c|}
\hline
$x$ & $H$ & $\omega$ & $\theta$ & $C_0$ & $C$ & $C'$ & $h'$ \\\hline
2 & 6.11  & 4.0067   &  1.0852   &  2.3688  &  0.2341  &0.51 & $\log{n}$+2.3974   \\\hline
3 & 6.52  & 4.0059   &  1.0797   &  2.2241  &  0.2198  &0.50 & $\log{n}$+2.1409   \\\hline
6 & 7.16  & 4.0049   &  1.0723   &  2.0867  &  0.2062  &0.50 & $\log{n}$+1.8139   \\\hline
7 & 7.29  & 4.0047   &  1.0710   &  2.0662  &  0.2042  &0.50 & $\log{n}$+1.7537   \\\hline
10& 7.59  & 4.0044   &  1.0681   &  2.0271  &  0.2003  &0.50 & $\log{n}$+1.6270   \\\hline
11& 7.67  & 4.0043   &  1.0674   &  2.0182  &  0.1994  &0.50 & $\log{n}$+1.5956   \\\hline
\end{tabular}
\vspace{0.2cm} \caption{Lower bounds for $H$ and upper bounds for $\omega,\theta,C_0,C,C',h'$ occurring in Lemma \ref{laurentlemma} if $y>4x^2$} \label{table3}
\end{table}
%\noindent Now, on combining \eqref{a1}, \eqref{a2} and $y>4x^2$ for $x \in A$ with Table \ref{table3} we get
%\begin{equation} \label{C_dash}
%\frac{\log{4}+\log(C'a_1a_2)}{\log{y}}<4.
%\end{equation}
By Lemma \ref{laurentlemma} we obtain
\begin{equation} \label{lambda_lower}
\log\vert \varLambda_r\vert>-Ch'^2a_1a_2-\sqrt{\omega\theta}h'-\log(C'h'^2a_1a_2),
\end{equation}
whence on comparing \eqref{lambda_upper} with \eqref{lambda_lower} we get
\begin{equation} \label{lambda_bound}
n<\left(\frac{Ch'^2a_1a_2}{\log{y}}+\frac{\sqrt{\omega\theta}}{\log{y}}h'+\frac{\log(C'h'^2a_1a_2)}{\log{y}}+\frac{\log{4}}{\log{y}}\right)\frac{\log{2x}}{\log{\frac{2x}{2x-1}}}.
\end{equation}
Finally, using \eqref{a2}, \eqref{a1} and $y>4x^2$ for $x \in A$, by Table \ref{table3} we easily see that inequality \eqref{lambda_bound} contradicts  \eqref{n_lower1}, proving the desired bounds for $n$ in this case.

\vskip.2cm\noindent {\it Case II. $y >10^6$ }

\vskip.2cm\noindent We work as in the previous case. Namely, we apply Lemma \ref{laurentlemma} again, the only difference is that in this case for $y$ we may write $y>10^6$. We may suppose, without loss of generality, that $n$ is large enough, that is
\begin{equation} \label{n_lower}
n>n_1.
\end{equation}

\noindent Further, we choose $\mu=0.57$ uniformly, and set
\begin{equation} \label{rho_y_10_to_6}
\rho=\begin{cases}
9.6&\mbox{if} \quad x=2,3,6,7 \\
9.3&\mbox{if} \quad x=10,11. \\
\end{cases}
\end{equation}

\noindent As before, we may take $a_1$ and $a_2$ as in \eqref{a1} and \eqref{a2}.

%\begin{equation} \label{a1_y_10_to_6}
%a_1=\begin{cases}
%\frac{\rho+1}{2}\log{x}+2\log{y}&\mbox{if} \quad x \in A \setminus\{22\} , \\
%\frac{\rho+1}{2}\log{x}+2\log{y}-2\log{11}&\mbox{if} \quad x=22. \\
%\end{cases}
%\end{equation}

%\noindent and

%\begin{equation}\label{a2_y_10_to_6}
%a_2=(\rho+1)\log{x}.
%\end{equation}

\noindent Thus by \eqref{rho_y_10_to_6}, \eqref{a1}, \eqref{a2} and $y>10^6$ for the values of $h$ occurring in Lemma \ref{laurentlemma} we obtain $h=\log{n}+\varepsilon$, with $\varepsilon=\varepsilon(x)$ given in Table \ref{table4}.

\begin{table}[h]
\centering
\begin{tabular}{|c|c|c|c|c|c|c|}
\hline
$x$ & $2$ & $3$ & $6$ & $7$ & $10$ & $11$ \\\hline
$\varepsilon$ & $0.0324$ & $-0.1870$ & $-0.4620$ & $-0.5122$ & $-0.6079$ & $-0.6339$ \\\hline\hline
\end{tabular}
\vspace{0.2cm} \caption{Choosing the parameter $h=\log{n}+\varepsilon$ occurring in Lemma \ref{laurentlemma} if $y > 10^6$} \label{table4}
\end{table}

%\begin{table}[htb]
%\centering
%\begin{tabular}{|c|c|}
%\hline
%$x$ & $h$      \\\hline
%5   & $\log{n}-0.0938$  \\\hline
%6   & $\log{n}-0.1851$  \\\hline
%9   & $\log{n}-0.3572$  \\\hline
%10  & $\log{n}-0.3964$  \\\hline
%13  & $\log{n}-0.4866$ \\\hline
%14  & $\log{n}-0.5103$ \\\hline
%17  & $\log{n}-0.5662$ \\\hline
%18  & $\log{n}-0.5795$ \\\hline
%21  & $\log{n}-0.6195$ \\\hline
%22  & $\log{n}-0.5866$ \\\hline
%\end{tabular}
%\vspace{0.2cm} \caption{Choosing the parameter $h$ occurring in Lemma \ref{laurentlemma} if $y >10^6$} \label{table4}
%\end{table}

%\newpage

\noindent On combining \eqref{a2}, \eqref{a1}, \eqref{n_lower}, \eqref{rho_y_10_to_6} with $y>10^6$ and with Table \ref{table4}
we obtain Table \ref{table5}.
\newpage
\begin{table}[h]
\centering
\begin{tabular}{|c|c|c|c|c|c|c|c|}
\hline
$x$ & $H$ & $\omega$ & $\theta$ & $C_0$ & $C$ & $C'$ & $h'$ \\\hline
2 & 5.04  & 4.0099   &  1.1042   &  2.1846  &  0.1587  & 0.36 & $\log{n}$+2.2943  \\\hline
3 & 5.49  & 4.0083   &  1.0953   &  2.1067  &  0.1531  & 0.36 & $\log{n}$+2.0749  \\\hline
6 & 6.17  & 4.0066   &  1.0844   &  2.0259  &  0.1472  & 0.36 & $\log{n}$+1.7999  \\\hline
7 & 6.31  & 4.0063   &  1.0824   &  2.0130  &  0.1463  & 0.36 & $\log{n}$+1.7497  \\\hline
10& 6.71  & 4.0056   &  1.0773   &  1.9871  &  0.1506  & 0.36 & $\log{n}$+1.6222  \\\hline
11& 6.79  & 4.0055   &  1.0764   &  1.9813  &  0.1502  & 0.36 & $\log{n}$+1.5962  \\\hline
\end{tabular}
\vspace{0.2cm} \caption{Lower bounds for $H$ and upper bounds for $\omega,\theta,C_0,C,C',h'$ occurring in Lemma \ref{laurentlemma} if $y >10^6$} \label{table5}
\end{table}

%\noindent Now, on combining \eqref{a2}, \eqref{a1} and $y>10^6$ with Table \ref{table5} we get
%\begin{equation} \label{C_dash_y_10_to_6}
%\frac{\log{4}+\log(C'a_1a_2)}{\log{y}}<4.
%\end{equation}
By Lemma \ref{laurentlemma} we obtain
\begin{equation} \label{lambda_lower_y_10_to_6}
\log\vert \Lambda_r\vert>-Ch'^2a_1a_2-\sqrt{\omega\theta}h'-\log(C'h'^2a_1a_2),
\end{equation}
whence on comparing \eqref{lambda_upper} with \eqref{lambda_lower_y_10_to_6} we obtain
\begin{equation} \label{lambda_bound_y>10-6}
n<\left(\frac{Ch'^2a_1a_2}{\log{y}}+\frac{\sqrt{\omega\theta}}{\log{y}}h'+\frac{\log(C'h'^2a_1a_2)}{\log{y}}+\frac{\log{4}}{\log{y}} \right)\frac{\log{2x}}{\log{\frac{2x}{2x-1}}}.
\end{equation}
Finally, using \eqref{a2}, \eqref{a1} and $y>10^6$, by Table \ref{table5} we see that \eqref{lambda_bound_y>10-6} contradicts \eqref{n_lower}, proving the validity of the desired bounds for $n$ in this case.

\vskip.2cm\noindent {\it Case III. $y \le 4x^2$ }

\vskip.2cm\noindent In order to obtain the desired upper bounds for $k$ we may clearly assume that $k$ is large, namely
\begin{equation} \label{k_lower}
k>k_1.
\end{equation}

%=k_2(x)=\begin{cases}
%77,500&\mbox{if} \quad x=5, \\
%120,500&\mbox{if} \quad x=6, \\
%303,500&\mbox{if} \quad x=9, \\
%380,500&\mbox{if} \quad x=10, \\
%650,500&\mbox{if} \quad x=13, \\
%753,500&\mbox{if} \quad x=14, \\
%1,098,500&\mbox{if} \quad x=17, \\
%1,223,500&\mbox{if} \quad x=18, \\
%1,632,500&\mbox{if} \quad x=21, \\
%1,465,500&\mbox{if} \quad x=22. \\
%\end{cases}
%\end{equation}

\noindent Since $y \le 4x^2$ we have by \eqref{eq:1.5} that
\begin{equation} \label{n_lower_k}
n>\lfloor k/2 \rfloor.
\end{equation}

\noindent Hence by \eqref{n_lower_k}, we can write
\begin{equation} \label{n_with_B_and_k}
n=Bk+r \quad \text{with} \quad B \ge 1, \ 0 \le |r| \le \left\lfloor \frac{k}{2} \right\rfloor.
\end{equation}

\noindent Further, using the same argument as in Case I, by $x \in A$ and $k \ge 83$ we may suppose that in \eqref{n_with_B_and_k} we have
$r \ne 0$.

%We show that in \eqref{n_with_B_and_k} $r \neq 0$.
%On the contrary, suppose $r=0$.
%Then we have $(B,k)=(1,n)$ (as $k$ is large and $n$ is a prime).
%Since
%\begin{align*}
%2(x-1)^{k}&>1+2^k+\cdots +(x-1)^k\\
%&=y^n-x^k \\
%&=y^{k}-x^{k} \\
%&=(y-x)(y^{k-1}+\cdots + x^{k-1}) \ge  x^{k-1},
%\end{align*}
%we have $2(x-1)^{k} \ge x^{k-1}$, or
%$$
%k \le \frac{\log (2x)}{ \log \bigr( \frac{x}{x-1}\bigr)} \ (<83),
%$$
%which is absurd by $x \in A$ and $k \ge 83$. Hence, $r \neq 0$.

\noindent We divide our equation \eqref{eq:1.5} by $(2x)^k$. Then, by \eqref{n_with_B_and_k} we infer
\begin{equation} \label{divide3}
y^{r}\left( \frac{y^{B}}{2x} \right)^{k}-1=\frac{s}{x^k},
\end{equation}
where $s=(x+1)^k+2^k+\cdots +(2x-1)^k$. Thus, $y^{r}\left( \frac{y^{B}}{2x} \right)^{k}>1$.
Put
\begin{equation} \label{lambda_k}
\varLambda_{r}=b_{2}\log\alpha_{2}-b_{1}\log \alpha_{1},
\end{equation}
\noindent where
\begin{equation} \label{lambda_k_choose}
(\alpha_1, \alpha_2, b_1,b_2)=
\begin{cases}
\left(\frac{2x}{y^B},y,k,r\right) & \text{if $r>0$},\\
\left(\frac{y^B}{2x},y,k,|r|\right) & \text{if $r<0$}.
\end{cases}
\end{equation}
It is easy to see $\alpha_1>1$ and $\alpha_2>1$, moreover similarly to {Case I} we obtain that $\alpha_1$ and $\alpha_2$ are multiplicatively independent.
We find upper and lower bounds for $\log |\varLambda_{r}|$. Since for every $z \in \mathbb{R}$ with $z>1$ we have $\vert \log{z}\vert<\vert z-1\vert$ it follows by \eqref{divide3}, \eqref{lambda_k}, \eqref{lambda_k_choose} and \eqref{skx_upper_2} that

%$$
%e^{\varLambda_{r}}-1
%=
%\begin{cases}
%\ y^{r} \cdot \left( y^{B}/x \right)^{k}-1=\frac{s}{x^k} & \text{if $r>0$},\\
%\ y^{-r} \cdot \left(x/y^B\right)^k-1=\frac{-s}{e^{\varLambda_{r}} x^k} & \text{if $r<0$}.
%\end{cases}
%$$
%$$
%\begin{cases}
%\varLambda_{r}>0 & \text{if $r>0$},\\
%\varLambda_{r}<0 & \text{if $r<0$}.
%\end{cases}
%$$
%Further, one can easily check that $x>y^B$ if $r>0$, and $y^B>x$ if $r<0$.

%\[
%|\varLambda_r| \le e^{|\varLambda_r|}-1=y^{r} \cdot \left(\frac{y^B}{x}\right)^k-1 =\frac{s}{x^k}, %<\frac{(x-1)^{k+1}}{y^n}.
%\]
%we have an upper bound:

\begin{equation}\label{lambda_upper_k}
\log |\varLambda_r|<-k\,\log \left(\frac{2x}{2x-1}\right) +\log 2.
\end{equation}
For a lower bound, we again use Lemma \ref{laurentlemma}. We choose $\mu=0.57$ uniformly, and we set for every $x \in A$
\begin{equation} \label{rho_k}
\rho=6.2.
\end{equation}

\noindent Moreover, using the same argument as in {Case I} by $y \le 4x^2$ we may take
\begin{equation} \label{a1_k}
a_1=1.02 \cdot (\rho+3)\log(2x),
\end{equation}
and
\begin{equation} \label{a2_k}
a_2=2\cdot(\rho+1)\log{2x}.
\end{equation}

\noindent Since $\mu=0.57$ we get
\begin{equation} \label{sigma_lam_k}
\sigma=0.90755 \ \textrm{and} \ \lambda=0.90755\log{\rho},
\end{equation}
whence by \eqref{rho_k}, \eqref{a1_k}, \eqref{a2_k},  \eqref{sigma_lam_k} we easily check that for every $x \in A$
$$
a_1a_2>\lambda^2
$$
holds. Now, we are going to derive an upper bound $h$ for the quantity
$$
\max \left\lbrace
D \left( \log \left( \frac{b_{1}}{a_{2}}+ \frac{b_{2}}{a_{1}} \right) + \log \lambda +1.75 \right)+0.06,\,
\lambda, \,
\frac{D \log 2}{2} \right\rbrace.
$$

\noindent Using  \eqref{rho_k}, \eqref{a1_k}, \eqref{a2_k}, \eqref{sigma_lam_k} for $h$ occurring in Lemma \ref{laurentlemma} we obtain $h=\log{n}+\varepsilon$, with $\varepsilon=\varepsilon(x)$ given in Table \ref{table6}.

\begin{table}[h]
\centering
\begin{tabular}{|c|c|c|c|c|c|c|}
\hline
$x$ & $2$ & $3$ & $6$ & $7$ & $10$ & $11$ \\\hline
$\varepsilon$ & $-0.1099$ & $-0.3665$ & $-0.6935$ & $-0.7537$ & $-0.8805$ & $-0.9118$ \\\hline
\end{tabular}
\vspace{0.2cm} \caption{Choosing the parameter $h=\log{n}+\varepsilon$ occurring in Lemma \ref{laurentlemma} if the case $y \le 4x^2$} \label{table6}
\end{table}

%\begin{table}[htb]
%\centering
%\begin{tabular}{|c|c|}
%\hline
%$x$ & $h$      \\\hline
%5   & $\log{k}+0.1216$  \\\hline
%6   & $\log{k}+0.0112$  \\\hline
%9   & $\log{k}-0.1928$  \\\hline
%10  & $\log{k}-0.2397$  \\\hline
%13  & $\log{k}-0.3507$ \\\hline
%14  & $\log{k}-0.3792$ \\\hline
%17  & $\log{k}-0.4502$ \\\hline
%18  & $\log{k}-0.4702$ \\\hline
%21  & $\log{k}-0.5221$ \\\hline
%22  & $\log{k}-0.3945$ \\\hline
%\end{tabular}
%\vspace{0.2cm} \caption{Choosing the parameter $h$ occurring in Lemma \ref{laurentlemma} if $y \le x^2$} \label{table6}
%\end{table}

\noindent On combining \eqref{k_lower}, \eqref{rho_k}, \eqref{a1_k}, \eqref{a2_k}, \eqref{sigma_lam_k} with Table \ref{table6}
we obtain Table \ref{table7}.

\begin{table}[h]
\centering
\begin{tabular}{|c|c|c|c|c|c|c|c|}
\hline
$x$ & $H$ & $\omega$ & $\theta$ & $C_0$ & $C$ & $C'$ & $h'$ \\\hline
2 & 7.50  & 4.0045   &  1.0689   &  2.1294  &  0.2947  & 0.67& $\log{k}$+1.7145  \\\hline
3 & 7.94  & 4.0040   &  1.0650   &  2.0435  &  0.2828  & 0.67& $\log{k}$+1.4580  \\\hline
6 & 8.65  & 4.0034   &  1.0595   &  1.9620  &  0.2715  & 0.67& $\log{k}$+1.1309  \\\hline
7 & 8.80  & 4.0033   &  1.0585   &  1.9498  &  0.2698  & 0.67& $\log{k}$+1.0708  \\\hline
10& 9.13  & 4.0030   &  1.0563   &  1.9265  &  0.2666  & 0.67& $\log{k}$+0.9440  \\\hline
11& 9.23  & 4.0030   &  1.0557   &  1.9212  &  0.2659  & 0.67& $\log{k}$+0.9127  \\\hline
\end{tabular}
\vspace{0.2cm} \caption{Lower bounds for $H$ and upper bounds for $\omega,\theta,C_0,C,C',h'$ occurring in Lemma \ref{laurentlemma} if $y \le 4x^2$} \label{table7}
\end{table}

Further, on using Table \ref{table7} and Lemma \ref{laurentlemma} we obtain
\begin{equation} \label{lambda_lower_k}
\log\vert \Lambda_r\vert>-Ch'^2a_1a_2-\sqrt{\omega\theta}h'-\log(C'h'^2a_1a_2),
\end{equation}
whence, on comparing \eqref{lambda_upper_k} with \eqref{lambda_lower_k} we get
\begin{equation*}
k<\frac{Ch'^2a_1a_2+\sqrt{\omega\theta}h'+\log(2C'h'^2a_1a_2)}{\log \left(\frac{2x}{2x-1}\right)}.
\end{equation*}
Finally, using \eqref{rho_k}, \eqref{a1_k}, \eqref{a2_k}, by Table \ref{table7} we obtain the desired bounds for $k$ in this case.
Thus our lemma is proved.

\end{proof}

%################################%
\section{Formulas for $v_{2}(T_{k}(x)), v_{3}(T_{k}(x))$} \label{sec:4}
%################################%
For the proofs of our main results, we will need formulas for $v_{2}(T_{k}(x))$ and $v_{3}(T_{k}(x))$. The heart of the proof of Lemma \ref{lem:4.2} is the following lemma
\begin{lemma}\label{lem:4.1}
For $q,k,t\geq1 $ and $q\equiv 1 \pmod{2}$, we have
\begin{equation*}
v_{2}(T_{k}(2^tq))= \left\{ \def\arraystretch{1.2}
\begin{array}{@{}c@{\quad}l@{}}
t-1, & \text{if $k=1$ or $k$ is even,}\\
2t-2, & \text{if $k\geq 3$ is odd.}\\
\end{array} \right.
\end{equation*}
\end{lemma}

\begin{proof}
We shall follow the proof of Lemma $1$ of Macmillian-Sondow \cite{MS}. We induct on $t$. Now we introduce the following equality
\begin{equation} \label{eq:4.18}
 T_{k}(x)= S_{k}(2x)- S_{k}(x),
\end{equation}which we will use frequently on this work.
Since $S_{k}(2^2q)$ is even and $S_{k}(2q)$ is odd, by using \eqref{eq:4.18}.
we get $v_{2}(T_{k}(2q))=0$ and so Lemma \ref{lem:4.1} holds for $t=1$. By Lemma \ref{lem:3.2} with $x=2^tq$, it also holds for all $t\geq 1$ when $k=1$. Now we assume inductively that \eqref{eq:4.18} is true for fixed $t\geq 1$.\\

Let $m$ be a positive integer, we can write the power sum $S_{k}(2m)$ as
\begin{equation} \label{eq:4.19}
\begin{aligned}
S_{k}(2m)= m^{k}+ \sum_{j=1}^{m}((m-j)^k+(m+j)^k)
=&m^{k}+2\sum_{j=1}^{m}\sum_{i=0}^{[\frac{k}{2}]}{\binom{k}{2i}}m^{k-2i}j^{2i}\\
=&m^{k}+2\sum_{i=0}^{[\frac{k}{2}]}{\binom{k}{2i}}m^{k-2i}S_{2i}(m).\\
\end{aligned}
\end{equation}

By \eqref{eq:4.18}, putting $x=m$, we have
\begin{equation} \label{eq:4.20}
T_{k}(m)= S_{k}(2m)-S_{k}(m)
\end{equation}

Now we consider \eqref{eq:4.20} with $m=2^tq$. If $k\geq 2$ is even, we extract the last terms of the summations of $S_{k}(2m)$ and $S_{k}(m)$, then we can write as

\begin{equation*}
S_{k}(2^{t+1}q)=2^{kt}q^k+ 2^t \frac{S_{k}(2^tq)}{2^{t-1}}+ 2^{2t+1}\sum_{i=0}^{\frac{k-2}{2}}{\binom{k}{2i}}2^{t(k-2i-2)}q^{k-2i}S_{2i}(2^tq)
\end{equation*}
and
\begin{align*}
S_{k}(2^tq)=2^{k(t-1)}q^k+ 2^{t-1} \frac{S_{k}(2^{t-1}q)}{2^{t-2}}&&\\
+2^{2t-1}\sum_{i=0}^{\frac{k-2}{2}}{\binom{k}{2i}}2^{(t-1)(k-2i-2)}q^{k-2i}S_{2i}(2^{t-1}q).
\end{align*}
Hence we have
\begin{eqnarray*}
T_{k}(2^tq)=2^{k(t-1)}q^k(2^k-1)+2^{t-1}[2\frac{S_{k}(2^{t}q)}{2^{t-1}}- \frac{S_{k}(2^{t-1}q)}{2^{t-2}}]\\
                  +2^{2t+1}\sum_{i=0}^{\frac{k-2}{2}}{\binom{k}{2i}}2^{t(k-2i-2)}q^{k-2i}S_{2i}(2^tq)\\
                  -2^{2t-1}\sum_{i=0}^{\frac{k-2}{2}}{\binom{k}{2i}}2^{(t-1)(k-2i-2)}q^{k-2i}S_{2i}(2^{t-1}q).
\end{eqnarray*}
By the induction hypothesis, the fraction is actually an odd integer. Since $k(t-1)> t-1$, we get that $v_{2}(T_{k}(2^tq))=t-1$, as desired.\\

Now we consider the case $k\geq 3$ is odd. Similarly to the former case, we have
\begin{eqnarray*}
S_{k}(2^{t+1}q)= 2^{tk}q^k+ 2^{2t}kq\frac{S_{k-1}(2^tq)}{2^{t-1}}+ 2^{3t+1}\sum_{i=0}^{\frac{k-3}{2}}{\binom{k}{2i}}2^{t(k-2i-3)}q^{k-2i}S_{2i}(2^tq)
\end{eqnarray*}
and
\begin{eqnarray*}
S_{k}(2^{t}q)= 2^{k(t-1)}q^k+ 2^{2t-2}kq\frac{S_{k-1}(2^{t-1}q)}{2^{t-2}}\\
+2^{3t-2}\sum_{i=0}^{\frac{k-3}{2}}{\binom{k}{2i}}2^{(t-1)(k-2i-3)}q^{k-2i}S_{2i}(2^{t-1}q).
\end{eqnarray*}

From here, we get
\begin{eqnarray*}
T_{k}(2^tq)=2^{k(t-1)}q^k(2^k-1)+ 2^{2t-2}kq[ \frac{2^2S_{k-1}(2^{t}q)}{2^{t-1}}- \frac{S_{k-1}(2^{t-1}q)}{2^{t-2}}]\\
+2^{3t+1}\sum_{i=0}^{\frac{k-3}{2}}{\binom{k}{2i}}2^{t(k-2i-3)}q^{k-2i}S_{2i}(2^tq)\\
-2^{3t-2}\sum_{i=0}^{\frac{k-3}{2}}{\binom{k}{2i}}2^{(t-1)(k-2i-3)}q^{k-2i}S_{2i}(2^{t-1}q).
\end{eqnarray*}
Again by induction, the fraction is an odd integer.\\

Since $k(k-1)>2(t-2)$ and $k$ and $q$ are odd, wee see that $v_{2}(T_{k}(2^tq))=2t-2$, as required. This completes the proof of Lemma.
\end{proof}

\begin{lemma}\label{lem:4.2}
$(i)$ Let $x$ be a positive even integer. Then we have,
\begin{equation*}
v_{2}(T_{k}(x)) =
\left\{ \def\arraystretch{1.2}
\begin{array}{@{}c@{\quad}l@{}}
v_{2}(x)-1, & \text{if $k=1$ or k is even,}\\
2v_{2}(x)-2, & \text{if $k\geq 3$  is odd.}\\
\end{array}\right.
\end{equation*}

$(ii)$ Let $x$ be a positive odd integer. If $x$ is odd and $k=1$, then for any solution $(k,n,x,y)$ of \eqref{eq:1.5} we get $v_{2}(T_{k}(x))=v_{2}(3x+1)-1$ .\\

If  $ x \equiv 1,5 \pmod{8} $ and $ x \not\equiv 1 \pmod{32} $ with $k\neq 1$, then we have
\begin{equation*}
v_{2}(T_{k}(x))= \left\{ \def\arraystretch{1.2}
\begin{array}{@{}c@{\quad}l@{}}
v_{2}(7x+1)-1, & \text{if $x \equiv 1\pmod{8}$ and k=2,}\\
v_{2}((5x+3)(3x+1))-2, & \text{if $x \equiv 1 \pmod{8}$ and $k=3$,}\\
v_{2}(3x+1), & \text{if $x \equiv 5 \pmod{8}$ and $k\geq 3$ is odd,}\\
$1$, & \text{if $x \equiv 5 \pmod{8}$ and $k\geq 2$ is even,} \\
$2$, & \text{if $x \equiv 9 \pmod{16}$ and $k\geq 4$ is even,}\\
$3$, & \text{if $x \equiv 9 \pmod{16}$ and $k\geq 5$ is odd}\\
& \text{or}\\
&\text{if $x \equiv 17 \pmod{32}$ and $k\geq 4$ is even,} \\
$4$, & \text{if $x \equiv 17 \pmod{32}$ and $k\geq 5$ is odd.} \\
\end{array}
\right.
\end{equation*}
If  $ x \equiv 3,7 \pmod{8} $, then for any solution $(k,n,x,y)$ of \eqref{eq:1.5}, we obtain $v_{2}(T_{k}(x))=0.$
\end{lemma}

\begin{proof}
$(i)$ Firstly, if $k\geq 2$ is even, since $2x+1$ is always odd, then we have
\begin{equation*}
v_{2}(\frac{x(2x+1)}{2})=v_{2}(x)-1.
\end{equation*}
Putting $x=2^tq$ where $q$ is odd and $t\geq 1$, we get
\begin{equation*}
v_{2}(\frac{2^tq(2^{t+1}q+1)}{2})=v_{2}(2^{t-1}q)=t-1.
\end{equation*}
Secondly if we consider the case $k\geq 3$ is odd, then
\begin{equation*}
v_{2}(\frac{x^2(3x+1)}{4})=v_{2}(x^2)-2.
\end{equation*}
Putting $x=2^tq$, we have
\begin{equation*}
v_{2}((2^tq)^2)-2=v_{2}(2^{2t})-2=2t-2.
\end{equation*}

Finally, in the case of $k=1$, we have
\begin{equation*}
v_{2}(\frac{x(3x+1)}{2})=v_{2}(x)-1.
\end{equation*}
Set $x=2^tq$, we have
\begin{equation*}
v_{2}(2^tq)-1=t-1.
\end{equation*}
So, the proof is completed.\\

$(ii)$ Since $S_{1}(x)=\frac{x(x+1)}{2}$, $S_{2}(x)=\frac{x(x+1)(2x+1)}{6}$ and $S_{3}(x)=(\frac{x(x+1)}{2})^{2}$ for any positive integer $x$, by \eqref{eq:4.18} if $x$ is odd or $x \equiv 1 \pmod{8}$, then the statement is automatic for $k=1$ or $k=2$, $3$, respectively.

Next we consider the case $x \equiv 5 \pmod{8}$ and $k\geq 3$ is odd. Since  $3x+1 \equiv 0 \pmod{8}$, we have $3x+1=2^dr$ with $d\geq 3$, $2\nmid r$. So we obtain
\begin{equation} \label{eq:4.21}
v_{2}(3x+1)=d
\end{equation}
Since $x$ is odd, $T_{k}(x)$ has  exactly odd terms. Putting $x=\frac{2^{d}r-1}{3}$  in  \eqref{eq:1.6}, we have

\begin{equation} \label{eq:4.22}
T_{k}(\frac{2^dr-1}{3}) = (\frac{1}{3})^k [(2^dr+2)^k + (2^dr+5)^k + \cdots + (32^{d-1}r)^k + \cdots + (2^{d+1}r-2)^k]
\end{equation}
which has $(32^{d-1}r)^k$ as the middle term of expansion. Considering \eqref{eq:4.22} in modulo $2^d$ with $k(d-1)>d$, we obtain $T_{k}(\frac{2^dr-1}{3})\equiv 0 \pmod{2^d}$. Then we have $v_{2}(T_{k}(\frac{2^dr-1}{3}))=v_{2}(2^dt)=d$  with $2\nmid t$. By \eqref{eq:4.21}, the statement follows in this case, as well.\\

Now we consider the case $x \equiv 5 \pmod{8}$ and $k\geq 2$ is even. We distinguish two cases. Assume first $k\geq 4$ is even. Using the polynomial

\begin{equation} \label{eq:4.23}
Q_{k}(x)=x^{k}+(x+1)^k+ (x+2)^k+ ... + 2^{k}(x-1)^k
\end{equation}
and the equality
\begin{equation} \label{eq:4.24}
T_{k}(x)=Q_{k}(x)-x^k+ (2x-1)^k+(2x)^k
\end{equation}
we obtain
\begin{equation*}
T_{k}(x) \equiv Q_{k}(x) \pmod{8}.
\end{equation*}
Then we have
\begin{equation} \label{eq:4.25}
v_{2}(T_{k}(x))= v_{2}(Q_{k}(x))
\end{equation}
Applying Lemma \ref{lem:4.2} $(i)$ on the polynomial  $Q_{k}(x)$ we obtain $v_{2}(Q_{k}(x))= v_{2}(x-1)-1$ and hence the statement follows also in this case. For the case $k=2$, by \eqref{eq:4.18}  we get also $v_{2}(T_{k}(x))=v_{2}(7x+1)-1=1$.\\

Next we consider the case  $x \equiv 9 \pmod{16}$ and $k\geq5$ is odd. By \eqref{eq:4.24} we have
\begin{equation} \label{eq:4.26}
T_{k}(x) \equiv Q_{k}(x)+8 \pmod{16}
\end{equation}
Using Lemma \ref{lem:4.2} $(ii)$, we have $v_{2}(Q_{k}(x))=2v_{2}(x-1)-2$. So, we get $v_{2}(Q_{k}(x))=4$ and
\begin{equation} \label{eq:4.27}
Q_{k}(x)=2^4t, \quad 2\nmid t
\end{equation}
By  \eqref{eq:4.26} and  \eqref{eq:4.27}, the statement follows in this case.\\

Now we consider the case  $x \equiv 9 \pmod{16}$ and $k\geq 4$ is even. By  \eqref{eq:4.24} we have
\begin{equation} \label{eq:4.28}
T_{k}(x) \equiv Q_{k}(x) \pmod{16}
\end{equation}
Using Lemma \ref{lem:4.2} $(i)$, we get $v_{2}(Q_{k}(x))=v_{2}(x-1)-1$ . So we get  $v_{2}(T_{k}(x))=v_{2}(Q_{k}(x))=2$ with \eqref{eq:4.28}.\\

Next we consider the case  $x \equiv 17 \pmod{32}$ and $k\geq 4$ is even. We distinguish two cases. If $k=4$ then,
\begin{equation} \label{eq:4.29}
T_{4}(x) \equiv Q_{4}(x)+16 \pmod{32}
\end{equation}
Using Lemma \ref{lem:4.2} $(i)$ we obtain $v_{2}(Q_{4}(x))=3$ and
\begin{equation} \label{eq:4.30}
Q_{4}(x)=2^3r, 2\nmid r
\end{equation}
By \eqref{eq:4.29} and  \eqref{eq:4.30} , we get $v_{2}(T_{4}(x))=3$. For the case $k\geq 6$ is even, by \eqref{eq:4.24} we have
\begin{equation*}
T_{k}(x) \equiv Q_{k}(x) \pmod{32}
\end{equation*}
Similar to the former cases, we obtain $v_{2}(T_{k}(x))=3$.\\

Now we consider $x \equiv 17 \pmod{32}$ and $k\geq 5$ is odd, by \eqref{eq:4.24} we have
\begin{equation} \label{eq:4.31}
T_{k}(x) \equiv Q_{k}(x)+16 \pmod{32}
\end{equation}
By  Lemma \ref{lem:4.2} $(ii)$, we have $v_{2}(Q_{k}(x))=6$ . With \eqref{eq:4.31} similar to the former cases, we get $v_{2}(T_{k}(x))=4$.\\

Next we consider the case  $x \equiv 3 \pmod{8}$. By \eqref{eq:4.24} we obtain
\begin{equation*}
T_{k}(x) \equiv Q_{k}(x)+2 \pmod{8}
\end{equation*}or
\begin{equation*}
T_{k}(x) \equiv Q_{k}(x) \pmod{8}
\end{equation*}
where $k$ is odd or even, respectively. In both cases we obtain  $v_{2}(Q_{k}(x))=0$ using  Lemma \ref{lem:4.2}. Then the statement follows in this case.\\

Now we consider the case  $x \equiv 7 \pmod{8}$. By \eqref{eq:4.24} we get
\begin{equation*}
T_{k}(x) \equiv Q_{k}(x)+6 \pmod{8}
\end{equation*} or
\begin{equation*}
T_{k}(x) \equiv Q_{k}(x)+2 \pmod{8}
\end{equation*}
where $k$ is odd or even, respectively. In both cases, we get $v_{2}(Q_{k}(x))=0$ using Lemma  \ref{lem:4.2}. Then the statement follows in this case, as well. So, the proof of Lemma is completed.
\end{proof}

\begin{lemma}\label{lem:4.3}
Assume that $k$ is not even if $x \equiv 5 \pmod{9}$. Then we have
\begin{equation*}
v_{3}(T_{k}(x))= \left\{\def\arraystretch{1.2}%
\begin{array}{@{}c@{\quad}l@{}}
v_{3}(x), & \text{if k=1,}\\
v_{3}(x)-1, & \text{if $x \equiv 0 \pmod{3}$ and $k\geq 2$ is even,}\\
v_{3}(kx^2), & \text{if $x \equiv 0 \pmod{3}$ and $k> 3$ is odd,}\\
v_{3}(x^2(5x+3)), & \text{if $x \equiv 0 \pmod{3}$ and $k= 3$,} \\
0, & \text{if $x \equiv \pm 1 \pmod{3}$ and $k\geq 3$ is odd,}\\
0, & \text{if $x \equiv 2,8 \pmod{9}$ and $k\geq 2$ is even,}\\
v_{3}(2x+1)-1, & \text{if $x \equiv 1 \pmod{3}$ and $k\geq 2$ is even.} \\
\end{array}
\right.
\end{equation*}
\end{lemma}

\begin{proof}
When $k=1$, $T_{1}(x)=\frac{x(3x+1)}{2}$. Then statement is shown automatically.\\

When $x \equiv 0 \pmod{3}$ and $k\geq 2$ is even, by \eqref{eq:3.17} we have
\begin{equation} \label{eq:4.32}
S_{k}(2x)\equiv 2S_{k}(x)\pmod{3^d}, \text{with p=3}.
\end{equation}
Considering \eqref{eq:4.18} in modulo  $3^{d}$, with \eqref{eq:4.32} we have
\begin{equation} \label{eq:4.33}
T_{k}(x)\equiv S_{k}(x) \pmod{3^d}.
\end{equation}
Using Lemma  \ref{lem:3.5} $(ii)$ and \eqref{eq:4.33}, we get $T_{k}(x)\equiv -3^{d-1}q \pmod{3^d}$. And hence $v_{3}(T_{k}(x))=d-1$. This is desired case.\\

When  $x \equiv 0 \pmod{3}$ and $k> 3$ is odd, writing $x=q3^d$ with $k=3^{\gamma}k'$ and $q\nmid 3$, by Lemma \ref{lem:3.4} we have
\begin{equation} \label{eq:4.34}
v_{3}(S_{k}(2x))= v_{3}(S_{k}(x))= \gamma+ 2d- 1
\end{equation}
Using \eqref{eq:4.18} and \eqref{eq:4.34}, we get
\begin{equation*}
T_{k}(x)\equiv 0 \pmod{3^{\gamma+ 2d}}.
\end{equation*}
And hence $v_{3}(T_{k}(x))= v_{3}(kx^{2})= \gamma+ 2d$.\\

When  $x \equiv 0 \pmod{3}$ and $k= 3$, we have $T_{3}(x)=\frac{x^2(5x+3)(3x+1)}{4}$. Since $3x+1 \equiv 1 \pmod{3}$, the statement follows in this case.\\

When  $x \equiv 1 \pmod{3}$ and $k\geq 3$ is odd, using Lemma \ref{lem:3.4}, $v_{3}(S_{k}(2x))= v_{3}(kx^2(x+1)^2)-1$ and $v_{3}(S_{k}(x))=0$.
By \eqref{eq:4.18} the statement follows in this case.\\

When $x \equiv 2 \pmod{3}$ and $k\geq 3$ is odd, using Lemma \ref{lem:3.4}, similar to the former case we obtain $v_{3}(T_{k}(x))=0$ with \eqref{lem:4.2}.\\

When $x \equiv 8 \pmod{9}$ or $x \equiv 2 \pmod{9}$ and $k\geq 2$ is even, by \eqref{eq:4.18} and Lemma \ref{lem:3.4} we get $v_{3}(S_{k}(2x))= v_{3}(2x(2x+1)(4x+1))-1$ or $v_{3}(S_{k}(2x))= v_{3}(x(x+1)(2x+1))-1$, respectively. If $x \equiv 8 \pmod{9}$, then $v_{3}(S_{k}(2x))=0$ and hence $v_{3}(T_{k}(x))=v_{3}(S_{k}(2x))$. If $x \equiv 2 \pmod{9}$, then $v_{3}(S_{k}(x))=0$ and hence $v_{3}(T_{k}(x))=v_{3}(S_{k}(x))$.\\

Assume now that $x \equiv 1 \pmod{3}$ and $k\geq 2$ is even. Applying Lemma \ref{lem:3.5} $(iv)$, with \eqref{lem:4.2} we obtain
\begin{equation*}
T_{k}(x)\equiv 3^{d-1}(-\frac{1}{2}) \pmod{3^{d}}.
\end{equation*}
And hence $v_{3}(T_{k}(x))=d-1$. By Lemma \ref{lem:3.5}, we write $x=q3^d+ r\frac{3^d-1}{2}$ where $r \equiv 1 \pmod{3}$, $0\leq q \not\equiv r \equiv x \pmod{3}$. So we get $2x+1=3^d(2q+1)$.
Since $v_{3}(2x+1)-1=d-1$, the statement follows in this case. So the proof is completed.
\end{proof}

%################################%
\section{Proofs of the main results} \label{sec:5}
%################################%

Now we are ready to prove our main results. We start with Theorem \ref{theo:2.1}, since it will be used in the proofs of the other statements.

\begin{proof}[Proof of Theorem \ref{theo:2.1}]
	
$(i)$ Since $x \equiv 0 \pmod{4}$, by Lemma \ref{lem:4.2} we have $v_{2}(T_{k}(x))> 0$, i.e $T_{k}(x)$ is even. Thus if \eqref{eq:1.5} satisfies, then $v_{2}(y)> 0$ and we have

\begin{equation*}
nv_{2}(y)=v_{2}(y^n)=v_{2}(T_{k}(x)) =
\left\{\def\arraystretch{1.2}
\begin{array}{@{}c@{\quad}l@{}}
v_{2}(x)-1, & \text{if $k$ is even,}\\
2v_{2}(x)-2, & \text{if $k$ is odd.}\\
\end{array}\right.
\end{equation*}
implying the statement in this case.\\

$(ii)$ As now $x \equiv 1,5 \pmod{8}$ and  $x \not\equiv 1 \pmod{32}$ with $k\neq 1$, Lemma \ref{lem:4.2} $(ii)$ implies that $v_{2}(T_{k}(x))> 0$. Hence \eqref{eq:1.5} gives $v_{2}(y)> 0$ and we have
\begin{equation*}
nv_{2}(y)=v_{2}(T_{k}(x))\\[10pt]
\\
=  \left\{\def\arraystretch{1.2}%
\begin{array}{@{}c@{\quad}l@{}}
v_{2}(7x+1)-1, & \text{if $x \equiv 1\pmod{8}$ and k=2,}\\
v_{2}((5x+3)(3x+1))-2, & \text{if $x \equiv 1 \pmod{8}$ and $k=3$,}\\
v_{2}(3x+1), & \text{if $x \equiv 5 \pmod{8}$ and $k\geq 3$ is odd,}\\
$1$, & \text{if $x \equiv 5 \pmod{8}$ and $k\geq 2$ is even,} \\
$2$, & \text{if $x \equiv 9 \pmod{16}$ and $k\geq 4$ is even,}\\
$3$, & \text{if $x \equiv 9 \pmod{16}$ and $k\geq 5$ is odd}\\
& \text{or}\\
& \text{if $x \equiv 17 \pmod{32}$ and $k\geq 4$ is even,} \\
$4$, & \text{if $x \equiv 17 \pmod{32}$ and $k\geq 5$ is odd.} \\
\end{array}
\right.
\end{equation*} \\
And if $x \equiv 1 \pmod{4}$ and $k=1$, then Lemma \ref{lem:4.2} $(ii)$ also implies that $v_{2}(T_{k}(x))> 0$. Hence \eqref{eq:1.5} gives $v_{2}(y)> 0$ and we obtain $nv_{2}(y)=v_{2}(y^n)=v_{2}(T_{k}(x))=v_{2}(3x+1)-1$. Implying the statement in this case, as well. So,the proof of the case $(ii)$ is completed.\\

$(iii)$ Suppose now that  $x \equiv 0 \pmod{3}$ and $k$ is odd or $x \equiv 0,1 \pmod{3}$ and $k\geq 2$ is even, by Lemma \ref{lem:4.3} implies that $v_{3}(y)> 0$ and we have
\begin{equation*}
nv_{3}(y)=v_{3}(T_{k}(x))\\
= \left\{\def\arraystretch{1.2}%
\begin{array}{@{}c@{\quad}l@{}}
v_{3}(x), & \text{if $x \equiv 0 \pmod{3}$ and k=1,}\\
v_{3}(x)-1, & \text{if $x \equiv 0 \pmod{3}$ and $k\geq 2$ is even,}\\
v_{3}(kx^2), & \text{if $x \equiv 0 \pmod{3}$ and $k>3$ is odd,}\\
v_{3}(x^2(5x+3)), & \text{if $x \equiv 0 \pmod{3}$ and $k=3$,} \\
v_{3}(2x+1)-1, & \text{if $x \equiv 1 \pmod{3}$ and $k\geq 2$ is even.} \\
\end{array}\right.
\end{equation*}\\
So, the proof of Theorem \ref{theo:2.1} is completed.
\end{proof}

\begin{proof}[Proof of Theorem \ref{theo:2.3}]
Observe that  since $x \equiv 4 \pmod{8}$, we have $v_{2}(T_{k}(x))=v_{2}(x)-1=1$. Hence if $k=1$ or $k$ is even then by part $(i)$ of Theorem \ref{theo:2.1} we obtain $n\leq 1,$ which is impossible. Since $x \equiv 5 \pmod{8}$, we have $v_{2}(T_{k}(x))=1$. Hence if $k\geq2$ is even then by part $(ii)$ of Theorem \ref{theo:2.1} we obtain $n\leq 1$, which is impossible. Since $x \equiv 1 \pmod{8}$, we have $v_{2}(T_{k}(x))=v_{2}(3x+1)-1=1$. Hence if $k=1$ then by part $(ii)$ of Theorem \ref{theo:2.1} we obtain $n\leq 1$, which is impossible. Thus, the proof is completed.
\end{proof}

\begin{proof}[Proof of Theorem \ref{theo:numerical}]
Let $2 \leq x \leq 13$ and consider equation \eqref{eq:1.5} in unknown integers $(k,y,n)$ with
$k \geq 1, y \geq 2$ and $n \geq 3$. We distinguish two cases according to $x \in \{ 2,3,6,7,10,11 \}$
or $x \in \{ 4,5,8,9,12,13\}$, respectively.

Assume first that $x \in \{ 2,3,6,7,10,11 \}$ is fixed. In this case for $k \leq 83$ a direct computation shows that
$T_k(x)$ is not a perfect $n^{\rm th}$ power, so equation \eqref{eq:1.5} has no solution.
Now we assume that $k\geq 83$. Now we split the treatment into 3 subcases according to the size of $y$.
If $y \leq 4x^2$ then Lemma \ref{L_Baker_bounds} shows that $k\leq k_1$.  Further, if $4x^2 < y \leq 10^6$ then we get
$n\leq n_0$ by Lemma \ref{L_Baker_bounds} and thus $T_k(x)\leq 10^{6n_0}$, which in turn gives
$$
k < \frac{6n_0\log10}{\log (2x)}.
$$
So for each $x$ under the assumption $y \leq 10^6$ we get a bound for $k$ and
we check for each $k$ below this bound and each $x\in \{ 2,3,6,7,10,11 \}$ if $T_k(x)$
has a prime factor $p$ with $p\leq y$. If not, then we are done, however, if such a $p$ exists, then we also
show, that for at least one such $p$ we have $\nu_p(T_k(x))\leq 12$, which shows that $n\leq 12$.
For $y<10^6$, $3\leq n\leq 12$ we get again very good bound for $k$ and a direct check will show that
equation \eqref{eq:1.5} has no solutions.

Now it is only left the case $y>10^6$, in which case we get $n<n_1$ by Lemma \ref{L_Baker_bounds},
and for each fixed $3\leq n\leq n_1$ we proceeded as follows.
Recall that $x \in \{ 2,3,6,7,10,11 \}$ is fixed, and we also fixed $3 \leq n \leq n_1$.
We took primes of the form $p:=2in+1$ with $i \in {\Bbb Z}$ and we considered equation \eqref{eq:1.5} locally modulo these
primes.
More precisely, we took the smallest such prime $p_1$ and put $o_1:=p_1-1$.
Then for all values of $k=1,\dots, o_1$ we checked whether $T_k(x) \pmod {p_1}$ is a perfect
power or not, and we built the set $K(o_1)$ of all those values of $k \pmod {o_1}$ for which $T_k(x) \pmod p$ was a perfect power. In principle this provided a list of all possible values of $k \pmod {o_1}$ for which we might have a solution.
Then we considered the next prime $p_2$ of the form $p_2:=2in+1$ with $i \in {\Bbb Z}$ and
we defined $o_2:=\textrm{LCM}(o_1, p_2-1)$. We expanded the set $K(o_1)$ to the set
$K_0(o_2)$ of all those numbers $1, \dots ,o_2$ which are congruent to
elements of $K(o_1)$ modulo $o_1$. Then we considered equation \eqref{eq:1.5} modulo $p_2$ and we excluded
from the set $K_0(o_2)$ all those elements $k$ for which $T_k(x) \pmod {p_2}$ is not a perfect power.
This way we got the set $K(o_2)$ of all possible values of $k \pmod {o_2}$ for which we might have a solution.
Continuing this procedure by taking new primes $p_3, p_4, \dots $ of the form $2in+1$ with $i \in {\Bbb Z}$, we
finished this procedure when the set $K(o_i)$ became empty, proving that equation \eqref{eq:1.5} has no solution for the given $x$ and $n$.

Suppose now that in equation \eqref{eq:1.5} we have $x \in \{ 4,5,8,9,12,13\}$. A direct application of Theorem \ref{theo:2.1}
to equation \eqref{eq:1.5} shows that for each $x \in \{ 4,5,8,9,12,13\}$ we may write $n \leq 5$. Finally, for every
$x \in \{ 4,5,8,9,12,13\}$ and $n \in \{3,4,5 \}$ we apply the same procedure as above in the case $y>10^6$ to conclude
that equation \eqref{eq:1.5} has no solution for the given $x$ and $n$.
This finishes the proof of our theorem.
\end{proof}

\noindent {\bf Remark.} The algorithms described in the above proof have been implemented in the
computer algebra package MAGMA \cite{MAGMA}. We mention that the running time of the
programme proving that we have no solution for
$x=11$ and $3 \leq n \leq n_1$ was more than 2 days on an Intel Xeon X5680 (Westmere EP) processor.
For $x=11$ to perform the computation up to the bound $n<n_0$ would have been too long.
This is the reason we had to use our bound $n_1$ proved in Lemma \ref{L_Baker_bounds} under the assumption
$y>10^6$.

%################################%
\subsection*{Acknowledgements}

The research was supported in part by the University of Debrecen, and by grants K115479 (A.B.) and NK104208 (A.B.) of the Hungarian National Foundation for Scientific Research. This paper was supported by the J\'anos Bolyai Scholarship of the Hungarian Academy of Sciences. The research was granted by the Austrian Science Fund (FWF) under the project P 24801-N26 (I.P.).
The fourth author was supported in part by the Scientific and Technological Research Council of Turkey (T\"{U}B\.{I}TAK) under 2219-International Postdoctoral Research Scholarship Program together with the Research Fund of Uluda\u{g} University under Project No: F-2015/23.


\begin{thebibliography}{99}
	
\bibitem{AS} {\sc M. Abramowitz, I. A. Stegun}, Handbook of Mathematical Functions, Dover, New York, (1965).

\bibitem{ZB} {\sc M. Bai, Z. Zhang}, On the Diophantine equation   $(x+1)^2+ (x+2)^2+ ... + (x+d)^2=y^{n}$, {\em Functiones Approx. Com. Math.} {\bf 49} (2013), 73-77.

\bibitem{BGP} {\sc M. A. Bennett, K. Gy\H{o}ry, \'{A}. Pint\'{e}r}, On the Diophantine equation $1^k+ 2^k+ ... + x^k=y^{n}$, {\em Compos. Math.} {\bf 140} (2004), 1417-1431.

\bibitem{BPS} {\sc M. A. Bennett, V. Patel, S. Siksek}, Superelliptic equations arising from sums of consecutive powers, {\em Acta Arith.} {\bf 172} (2016), 377-393.

\bibitem{BPS2} {\sc M. A. Bennett, V. Patel, S. Siksek}, Perfect powers that are sums of consecutive cubes, {\em Mathematika} {\bf 63} (2016), 230-249.
	
\bibitem{BHMP} {\sc A. B\'{e}rczes, L. Hajdu, T. Miyazaki, I. Pink}, On the equation $1^k+ 2^k+ ... +x^k=y^n,$ {\em J. Number Theory } {\bf 163} (2016), 43-60.

\bibitem{MAGMA} {\sc W. Bosma, J. Cannon, C. Playoust} The Magma algebra system. I. The user language, {\em J.~Symbolic Comput.} {\bf 24} (1997), 235–265.

\bibitem{GP} {\sc K. Gy\H{o}ry, \'{A}. Pint\'{e}r},On the equation $1^k+ 2^k+ ... +x^k=y^{n}$, {\em Publ. Math. Debrecen } {\bf 62} (2003), 403-414.

\bibitem{H} {\sc L. Hajdu}, On a conjecture of Sch\"{a}ffer concerning the equation   $1^k+ 2^k+ ... + x^k=y^n$ , {\em J. Number Theory } {\bf 155} (2015), 129-138.

\bibitem{JPW} {\sc M. Jacobson, \'{A}.Pint\'{e}r, G.P.Walsh,} A computational approach for solving $y^{2}=1^k+ 2^k+ ... +x^k$,
{\em Math. Comp.} {\bf 72} (2003), 2099-2110.

\bibitem{LAUR} {\sc M. Laurent}, Linear forms in two
logarithms and interpolation determinants II, {\em Acta Arith.} {\bf 133} (2008), 325-348.

\bibitem{Lu} {\sc \'{E}. Lucas}, Question $1180$, {\em Nouvelles Ann. Math} {\bf 14} (1875), 336.

\bibitem{MS} {\sc K. MacMillian, J. Sondow,} Divisibility of power sums and the generalized Erd\H{o}s-Moser equation, {\em Elem. Math.} {\bf 67} (2012), 182-186.

\bibitem{PS}{\sc V. Patel, S. Siksek,}{\em On powers that are sums of consecutive like powers}, {\em Research in Num. Theory} (2017), to appear.

\bibitem{P} {\sc \'{A}. Pint\'{e}r,} On the power values of power sums, {\em J. Number Theory } {\bf 125} (2007), 412-423.

\bibitem{Rd} {\sc H. Rademacher}, Topics in Analytic Number Theory, Springer-Verlag, Berlin, (1973).

\bibitem{Sch} {\sc J. J. Sch\"{a}ffer}, The equation $1^p+ 2^p+ ... +n^p =m^q$, {\em Acta Math.} {\bf 95} (1956), 155-189.

\bibitem{smart} {\sc N. P. Smart}, The Algorithmic Resolutions of Diophantine Equations, Cambridge University Press, Cambridge, (1998).

\bibitem{ST} {\sc J. Sondow, E.Tsukerman}, The p-adic order of power sums, the Erd\H{o}s-Moser equation and Bernoulli numbers,
{\em arXiv:1401.0322v1 [math.NT]}, 1 Jan 2014.

\bibitem{S} {\sc G. Soydan}, On the Diophantine equation   $(x+1)^k+ (x+2)^k+ ... + (lx)^k=y^{n}$, {\em Publ. Math. Debrecen } (2017), to appear.

\bibitem{Wat} {\sc G. N. Watson}, The problem of the square pyramid,
{\em Messenger of Math} {\bf 48} (1918), 1-22.

\bibitem{Zh} {\sc Z. Zhang}, On the Diophantine equation   $(x-1)^k+ x^k+(x+1)^k=y^{n}$, {\em Publ. Math. Debrecen } {\bf 85} (2014), 93-100.

\end{thebibliography}
\end{document}